\documentclass{article}

\usepackage{arxiv}
\usepackage{amsfonts}       
\usepackage{amsmath} 
\usepackage{amssymb}  
\usepackage{amsthm}
\usepackage{mathrsfs}
\usepackage[colorlinks=true, allcolors=black]{hyperref}
\DeclareMathAlphabet{\mathcal}{OMS}{cmsy}{m}{n}
\DeclareSymbolFont{largesymbols}{OMX}{cmex}{m}{n}
\usepackage[normalem]{ulem}

\newcommand{\mS}{\mathcal{S}}
\newcommand{\R}{\mathbb{R}}
\newcommand{\mbE}{\mathbb{E}}

\newcommand{\innerp}[1]{\langle #1 \rangle}

\newcommand{\expectw}[2]{\mathbb{E}_{#1}\left( #2 \right)}
\newcommand{\prob}[1]{\mathbb{P}\left( #1 \right)}

\newtheorem{definition}{Definition}[section]
\newtheorem{lemma}{Lemma}[section]
\newtheorem{thm}{Theorem}

\title{
A New Proof of Sub-Gaussian Norm Concentration Inequality}
\author{
	{\hspace{1mm}Zishun Liu} \\
	Georgia Institute of Technology\\
	Atlanta, GA 30332, US \\
	\texttt{zliu910@gatech.edu} \\
    \And
	{\hspace{1mm}Sam Power} \\
	University of Bristol\\
	Bristol, BS8 1QU, UK \\
	\texttt{sam.power@bristol.ac.uk} \\
	\And
	{\hspace{1mm}Yongxin Chen} \\
	Georgia Institute of Technology\\
	Atlanta, GA 30332, US \\
	\texttt{yongchen@gatech.edu} \\
}

\date{}

\renewcommand{\headeright}{}
\renewcommand{\undertitle}{}
\renewcommand{\shorttitle}{A New Proof of the Sub-Gaussian Norm Concentration Inequality}

\hypersetup{
	pdftitle={A New Proof of the Sub-Gaussian Norm Concentration Inequality},
	pdfsubject={math.PR, math.ST},
	pdfauthor={Zishun Liu, Sam Power, Yongxin Chen},
	pdfkeywords={ },
}

\begin{document}
\maketitle

\begin{abstract}  
    We present a new method for proving the norm concentration inequality of sub-Gaussian variables. Our proof is based on an averaged version of the moment generating function, termed the averaged moment generating function. Our method applies to both vector cases to bound the vector norm and matrix cases to bound the operator norm.
    Compared with the widely adopted $\varepsilon$-net technique-based proof of the sub-Gaussian norm concentration inequality, our method does not rely on the union bound and promises a tighter concentration bound.  
\end{abstract}

\section{Sub-Gaussian Norm Concentration Inequality}

Concentration inequalities are mathematical tools in probability theory that describe how a random variable deviates from some value, typically its expectation. Some commonly used instances include Markov's inequality, Chebyshev's inequality, and Chernoff bounds. These concentration inequalities play an essential role in various fields, including probability theory, statistics, machine learning, finance, etc, providing probabilistic guarantees when dealing with random quantities. 
Many probability distributions exhibit concentration properties, among which we consider an important class known as \textit{sub-Gaussian} distribution. The sub-Gaussian random variable is formally defined as follows.
\begin{definition}
A random variable $X\in\R$ is said to be sub-Gaussian with variance proxy $\sigma^2>0$ if
    \begin{equation}
        \mbE_X(e^{\lambda X})\leq e^{\frac{\lambda^2\sigma^2}{2}},~ \forall \lambda\in\R.
    \end{equation}
A random vector $X\in\R^n$ is sub-Gaussian with variance proxy $\sigma^2>0$ if $\innerp{\ell,X}$ is sub-Gaussian for any unit vector $\ell$, that is,
    \begin{equation}\label{eq: subG}
        \mbE_X(e^{\lambda\innerp{\ell,X}})\leq e^{\frac{\lambda^2\sigma^2}{2}},~ \forall \lambda\in\R,~\forall \ell\in\mS^{n-1},
    \end{equation}
    where $\mS^{n-1}=\{x\in\R^n:~ \|x\|=1\}$ denotes the Euclidean unit sphere in $\R^n$.
\end{definition}
Sub-Gaussian distributions include a wide range of distributions such as Gaussian distribution, uniform distribution, and any distributions with finite supports as special cases. For a Gaussian random variable, the variance proxy equals its variance. 

In this work, we consider one important concentration inequality for sub-Gaussian distributions known as norm concentration \cite{rigollet2023high,wainwright2019high}, which describes the concentration property of the norm $\|X\|$ for a sub-Gaussian random vector $X$.
\begin{thm} \label{thm: original norm concen}
     For a sub-Gaussian vector $X\in\R^n$ with variance proxy $\sigma^2$, there exist constants $C_1$, $C_2$ such that, for any $\delta\in(0,1)$,
\begin{equation}\label{eq: norm concentration e-net}
    \prob{\|X\|\leq \sigma\sqrt{C_1n+C_2\log\frac{1}{\delta}}} \ge 1-\delta.
\end{equation}
\end{thm}

The choice of constants $C_1, C_2$ depends on the proof techniques.
When $n=1$, a direct application of the Markov's inequality points to 
    \[
        \prob{X>r} = \prob{e^{\lambda X}>e^{\lambda r}} \le \frac{\mbE_X(e^{\lambda X})}{e^{\lambda r}} \le e^{\frac{\lambda^2\sigma^2}{2}-\lambda r}
    \]
for any $\lambda>0, r>0$. Minimizing the above over $\lambda$ yields $\prob{X>r} \le e^{-\frac{r^2}{2\sigma^2}}$. Similarly, $\prob{X<-r} \le e^{-\frac{r^2}{2\sigma^2}}$. Combining them with the union bound, we obtain $\prob{\|X\|>r} \le 2 e^{-\frac{r^2}{2\sigma^2}}$. This corresponds to \eqref{eq: norm concentration e-net} with constants $C_1=2\log2$ and $C_2=2$. 

When $n\geq2$, one can follow a similar idea and apply Markov's inequality to bound $\prob{\ell^\top X>r}$ for each $\ell\in\mS^{n-1}$ and then combine them to establish a probabilistic bound for $\|X\|=\max_{\ell\in\mS^{n-1}}\ell^\top X$. However, the union bound is no longer applicable since the maximum is over an unaccountable set $\mS^{n-1}$. To circumvent this issue, the $\varepsilon$-net technique has been developed and has become a standard proof for sub-Gaussian norm concentration \cite{rigollet2023high,vershynin2018high}. Following this technique, one constructs an $\varepsilon$-net $\mathcal{N}$ \cite[Definition 1.17]{rigollet2023high} for the Euclidean unit ball such that $\mathcal{N}$ has finite elements and $\max_{\ell\in\mS^{n-1}}\ell^\top X\leq \frac{1}{1-\varepsilon}\max_{\ell\in\mathcal{N}}\ell^{\top}X$ \cite[Exercise 4.4.2]{vershynin2018high}. The union bound can be applied since $\mathcal{N}$ is a finite set. This method results in \eqref{eq: norm concentration e-net} with constants $C_1$ and $C_2$ determined by 
\begin{equation}\label{eq: tC val}
    C_1=\frac{2\log(1+2/(1-\varepsilon))}{\varepsilon^2},~\,\, C_2=\frac{2}{\varepsilon^2}
\end{equation}
where $\varepsilon$ can take any value in the interval $(0,1)$.
When $\varepsilon=\frac{1}{2}$, we have $C_1=8\log5\approx16$ and $C_2=8$, which is commonly used in existing literature, e.g. \cite{rigollet2023high}. 

\section{New Proof based on Averaged Moment Generating Function} \label{sec: amgf}
To analyze the concentration property of $\|X\|$, the $\varepsilon$-net strategy as described above first analyzes the concentration of the one-dimensional projection $\ell^\top X$ of $X$ using the Markov's inequality and then applies the union bound to establish the concentration bound for $\|X\|$. This gives rise to one question: is there a more direct approach to establish the concentration of $\|X\|$? 

One natural attempt is to bound $\mbE_X (e^{\lambda \|X\|})$ with $\lambda>0$ and apply the Markov's inequality to get bound $\prob{\|X\|>r}\le \mbE_X(e^{\lambda \|X\|})/e^{\lambda r}$ directly. However, it is not easy to bound $\mbE_X (e^{\lambda \|X\|})$ from the definition \eqref{eq: subG} of the sub-Gaussian random vector. In fact, researchers often follow an opposite direction and bound $\mbE_X (e^{\lambda \|X\|})$ using the concentration property of $\|X\|$.

In this paper, we present a new proof of the sub-Gaussian norm concentration inequality that directly analyzes $\|X\|$. The key of our proof is a novel mathematical tool named the averaged moment generating function (AMGF) that bears similarity with $\mbE_X (e^{\lambda \|X\|})$.
\begin{definition}[AMGF]\label{def: amgf}
The Averaged Moment Generating Function $\Phi_X(\lambda)$ of a random vector $X\in\R^n$ is defined as $\Phi_X(\lambda)=\mbE_X(\Phi_{n}(\lambda X))$ where the energy function
    \begin{equation}\label{eq: AMGF}
        \Phi_{n}(\lambda X) = \expectw{\ell\sim\mS^{n-1}}{e^{\lambda\innerp{\ell,X}}}.
    \end{equation}
\end{definition}

The AMGF was first introduced in \cite{altschuler2022concentration} to study Langevin algorithms in sampling problems. The AMGF is an average of the moment generating function (MGF) $\mbE_X(e^{\lambda\innerp{\ell,X}})$ over the unit sphere $\ell\sim\mS^{n-1}$. It can also be viewed as an MGF with the exponential energy function $e^{\lambda\innerp{\ell,X}}$ being replaced by its average $\Phi_{n}(\lambda X)=\expectw{\ell\sim\mS^{n-1}}{e^{\lambda\innerp{\ell,X}}}$. 

Similar to $e^{\lambda \|X\|}$, the energy function $\Phi_{n}(\lambda X)$ only depends on the norm $\|X\|$. 
In fact,  
    \begin{equation}
    \Phi_n(\lambda X)=\phi_n(\|\lambda X\|),
    \end{equation}
where $\phi_n(z)=\expectw{\ell\sim\mS^{n-1}}{e^{\innerp{\ell,z\,\eta}}},~ \forall\eta\in\mS^{n-1}$ 
admits a closed-form expression of $\phi_1(z)= \cosh (z)$ and $\phi_n(z) = \Gamma(n/2)(2/z)^{(n-2)/2}I_{(n-2)/2}(z)$ when $n\ge 2$ where $\Gamma$ denotes the Gamma function and $I$ is the modified Bessel function of the first kind \cite{altschuler2022concentration}.
In addition to these similarities to $e^{\lambda \|x\|}$, we show below that the value of $\Phi_{n}(\lambda X)$ is closely related to $e^{\|\lambda X\|}$.

\begin{lemma}\label{lemma: AMGF_2}
For any $X\in\R^n$ and any $\varepsilon\in(0,1)$, $\Phi_n(\lambda X)$ satisfies
    \begin{equation}\label{eq: lemma amgf2}
        \Phi_n(\lambda X)\geq (1-\varepsilon^2)^{\frac{n}{2}}e^{\varepsilon\|\lambda X\|}.
    \end{equation}
\end{lemma}

\begin{proof}
In view of $ \Phi_n(\lambda X)=\phi_n(\|\lambda X\|)$, it suffices to show that 
\begin{equation}\label{eq:phin}
    \phi_n(z)\geq (1-\varepsilon^2)^{\frac{n}{2}}e^{\varepsilon z}
\end{equation} 
holds for any $z\geq0$ and any $\varepsilon\in(0,1)$. When $n=1$, $\phi_1(z)=\cosh(z)$ and \eqref{eq: lemma amgf2} becomes  
    \begin{equation}\label{eq: amgf lb n=1}
        \cosh(z)\geq \sqrt{1-\varepsilon^2}e^{\varepsilon z},~\forall z\geq0,~ \forall \varepsilon\in(0,1).
    \end{equation}
This is a straightforward consequence of the convexity of $\log\cosh(z)-\log\left(\sqrt{1-\varepsilon^2}e^{\varepsilon z}\right)$ with respect to $z$. We next focus on the case of $n\geq2$.

As pointed out in \cite{altschuler2022concentration}, directly from the definition, the derivative of $\log\phi_n$ admits an elegant expression
    \begin{equation}\label{eq: dlog_amgf}
    \frac{d}{dz}\log\phi_n(z)=\frac{I_{\frac{n}{2}}(z)}{I_{\frac{n}{2}-1}(z)}.
    \end{equation}
The ratio of Bessel functions can be bounded below as \cite{Amos1974Bessel}
    \begin{equation}\label{eq: bessel ratio lb}
        \frac{I_{\frac{n}{2}}(z)}{I_{\frac{n}{2}-1}(z)}\geq \sqrt{1+\left(\frac{n}{2z}\right)^2}-\frac{n}{2z}.
    \end{equation}
Denote $g(z)=\sqrt{1+(\frac{n}{2z})^2}-\frac{n}{2z}$ and $G(z)=\int_{0}^z g(y)dy$, then, by \eqref{eq: dlog_amgf}-\eqref{eq: bessel ratio lb},
\begin{equation}\label{eq: log_amgf>=}
    \log \phi_n(z)\geq G(z)+\log \phi_n(0)= G(z), \quad \forall z\ge 0. 
\end{equation}
Since $g(z)$ is monotonically increasing over $z>0$, $G(z)$ is convex on $z>0$. Thus, by the definition of convexity, given any $z_0>0$,
\begin{equation}\label{eq: convex G}
    G(z)\geq g(z_0)(z-z_0)+G(z_0).
\end{equation}
Set $z_0=\frac{\varepsilon n}{1-\varepsilon^2}$, then
\begin{equation}\label{eq:gz0}
    g(z_0)=\varepsilon,\quad G(z_0)=\frac{n\varepsilon^2}{1-\varepsilon^2}+\frac{n}{2}\log(1-\varepsilon^2).
\end{equation}
Plugging \eqref{eq:gz0} into \eqref{eq: log_amgf>=}-\eqref{eq: convex G} yields
\begin{equation}\label{eq: log_amgf lin lb}
    \log \phi_n(z)\geq \varepsilon z + \frac{n}{2}\log(1-\varepsilon^2).
\end{equation}
The conclusion \eqref{eq:phin} and thus \eqref{eq: lemma amgf2} follow by taking the exponential of \eqref{eq: log_amgf lin lb}. This completes the proof.
\end{proof}

By Lemma \ref{lemma: AMGF_2}, the AMGF can be viewed as a surrogate function of $\mbE_X (e^{\lambda \|X\|})$. One distinguishing feature of the AMGF compared with $\mbE_X (e^{\lambda \|X\|})$ is that the AMGF shares the same bound as the MGF in \eqref{eq: subG}. More specifically, by \eqref{eq: subG}, $\mbE_X(e^{\lambda\innerp{\ell,X}})\leq e^{\frac{\lambda^2\sigma^2}{2}}$ for any $\ell\in\mS^{n-1}$, and thus the AMGF satisfies
    \begin{equation}\label{eq:AMGFbound}
        \Phi_X(\lambda)=\mbE_X(\Phi_{n}(\lambda X))= \mbE_X\mbE_{\ell\sim\mS^{n-1}}(e^{\lambda\innerp{\ell,X}})=\mbE_{\ell\sim\mS^{n-1}}\mbE_X(e^{\lambda\innerp{\ell,X}})\leq e^{\frac{\lambda^2\sigma^2}{2}}.
    \end{equation}
By combining \eqref{eq:AMGFbound} with Lemma \ref{lemma: AMGF_2}, an upper bound on $\mbE_X (e^{\lambda \|X\|})$ can be established, based on which Markov's inequality can be applied to bound $\prob{\|X\|>r}$, formalized as follows. 

\begin{thm}\label{thm: norm concentration}
    Let $X\in\R^n$ be a sub-Gaussian vector with variance proxy $\sigma^2$, then for any $\delta\in(0,1)$ and any $\varepsilon\in(0,1)$,
    \begin{equation}\label{eq: thm1}
        \prob{\|X\|\leq \sigma\sqrt{\frac{\log\frac{1}{1-\varepsilon^2}}{\varepsilon^2} n+\frac{2}{\varepsilon^2}\log\frac{1}{\delta}}}\ge 1-\delta.
    \end{equation}
\end{thm}
\begin{proof}
    By Lemma \ref{lemma: AMGF_2}, the AMGF of $X$ satisfies
    \begin{equation} \label{eq: AMGF lower}
            \Phi_X(\lambda)\geq (1-\varepsilon^2)^{\frac{n}{2}}\expectw{X}{e^{\varepsilon\|\lambda X\|}}.
    \end{equation}
In view of \eqref{eq:AMGFbound}, we obtain
    \begin{equation}\label{eq: MGF upper}
        \expectw{X}{e^{\varepsilon\|\lambda X\|}}\leq (1-\varepsilon^2)^{-\frac{n}{2}}\exp\left(\frac{\lambda^2\sigma^2}{2}\right).
    \end{equation}
    Denote $t=\varepsilon|\lambda|>0$, then \eqref{eq: MGF upper} becomes
    \begin{equation} \label{eq: bound E(exp|X|)}
        \expectw{X}{e^{t\|X\|}}\leq (1-\varepsilon^2)^{-\frac{n}{2}}\exp\left(\frac{\sigma^2t^2}{2\varepsilon^2}\right).
    \end{equation}
    
    Applying Markov's inequality, we obtain that, for any $r>0$ and $\varepsilon\in(0,1)$,
    \begin{equation}\label{eq: markov}
        \begin{split}
            \prob{\|X\|\leq r} =&~ \prob{e^{t\|X\|}\leq e^{tr}}\geq 1-\frac{\expectw{X}{e^{t\|X\|}}}{e^{tr}} \\
            \geq&~ 1-(1-\varepsilon^2)^{-\frac{n}{2}}\exp(\frac{\sigma^2t^2}{2\varepsilon^2}- rt).
        \end{split}
    \end{equation}
Minimizing the exponent $\frac{\sigma^2t^2}{2\varepsilon^2}- rt$ over $t$ yields the minimum $-\frac{\varepsilon^2r^2}{2\sigma^2}$. Plugging the minimum into \eqref{eq: markov} we arrive at
\begin{equation}\label{eq: P(X<r)}
    \prob{\|X\|\leq r}\geq 1-(1-\varepsilon^2)^{-\frac{n}{2}}e^{-\frac{\varepsilon^2r^2}{2\sigma^2}}.
\end{equation}
 Define $\delta=(1-\varepsilon^2)^{-\frac{n}{2}}e^{-\frac{\varepsilon^2r^2}{2\sigma^2}}$, then $r$ can be expressed as
 $$r=\sigma\sqrt{\frac{\log\frac{1}{1-\varepsilon^2}}{\varepsilon^2}n+\frac{2}{\varepsilon^2}\log\frac{1}{\delta}}.$$ 
The conclusion \eqref{eq: thm1} follows by plugging this expression of $r$ into \eqref{eq: P(X<r)}. This completes the proof.
 \end{proof}

Clearly, Theorem \ref{thm: norm concentration} is a special case of Theorem \ref{thm: original norm concen} with constants
    \begin{equation}\label{eq:nC12}
        C_1 = \frac{\log\frac{1}{1-\varepsilon^2}}{\varepsilon^2},\quad C_2= \frac{2}{\varepsilon^2}.
    \end{equation}
The choice \eqref{eq:nC12} of $C_1, C_2$ appears to be better than \eqref{eq: tC val} obtained by the $\varepsilon$-net technique. Indeed, \eqref{eq:nC12} shares the same $C_2$ as \eqref{eq: tC val}, and $C_1$ in \eqref{eq:nC12} is smaller by
    \[
        \frac{2\log(1+\frac{2}{1-\varepsilon})}{\varepsilon^2}=\frac{\log(1+\frac{4}{1-\varepsilon}+\frac{4}{(1-\varepsilon)^2})}{\varepsilon^2}
        > \frac{\log(1+\frac{\varepsilon^2}{1-\varepsilon^2})}{\varepsilon^2}=\frac{\log(\frac{1}{1-\varepsilon^2})}{\varepsilon^2}
    \]
for any $\varepsilon \in (0,\,1)$.  

Theorem \ref{thm: norm concentration} represents a class of concentration inequalities parameterized by $\varepsilon \in (0,\,1)$. In practice, one can choose $\varepsilon$ according to the value of $n$ and $\delta$. With slight modifications, one can derive a more elegant concentration inequality based on AMGF as follows.
\begin{thm} \label{thm3}
    Let $X\in\R^{n}$ be a sub-Gaussian vector with variance proxy $\sigma^2$, then for any $\delta\in(0,1)$,
    \begin{equation}\label{eq:concen2}
        \prob{\|X\|\leq \sigma\left(\sqrt{n}+\sqrt{2\log(1/\delta)}\right)} \geq 1-\delta.
    \end{equation}
\end{thm}

\begin{proof}
   Combining \eqref{eq: bound E(exp|X|)} and the fact that $\log(1-\varepsilon^2)\geq \frac{\varepsilon^2}{\varepsilon^2-1}$ holds for any $\varepsilon\in(0,1)$, we know for any $\varepsilon\in(0,1)$ and $t>0$:
    \begin{equation}\label{eq:opteps}
        \expectw{X}{e^{t\|X\|}}
        \leq\min_{\varepsilon\in(0,1)}\exp\left(\frac{n\varepsilon^2}{2(1-\varepsilon^2)}+\frac{\sigma^2t^2}{2\varepsilon^2}\right).
    \end{equation}
By Cauchy's inequality:
    \begin{equation}\label{eq: MGF_t upper}
    \begin{split}
       \exp\left(\frac{n\varepsilon^2}{2(1-\varepsilon^2)}+\frac{\sigma^2t^2}{2\varepsilon^2}\right)
        =& \exp\left( \frac{1}{2}\left\lVert\begin{bmatrix} \frac{\sqrt{n}}{\sqrt{1-\varepsilon^2}} \\ \frac{\sigma t}{\varepsilon} \end{bmatrix}\right\rVert^2 \cdot \left\lVert\begin{bmatrix} \sqrt{1-\varepsilon^2} \\ \sqrt{\varepsilon^2} \end{bmatrix}\right\rVert^2 -\frac{n}{2}\right) \\
        \geq& \exp\left(\frac{1}{2}(\sqrt{n}+\sigma t)^2-\frac{n}{2}\right),
    \end{split}
    \end{equation}
    where the minimizer is $\varepsilon^*=\sqrt{\frac{\sigma t}{\sigma t+\sqrt{n}}}\in(0,1)$. By Markov's inequality, \eqref{eq: MGF_t upper} implies that for any $t>0$:
    \begin{equation}\label{eq: sub-G markov}
    \begin{split}
         &\prob{\|X\|\geq \sigma(\sqrt{n}+r)} \leq \frac{\expectw{X}{e^{t\| X\|}}}{e^{\sigma t(\sqrt{n}+r)}} \\
         \leq& \exp\left( \frac{1}{2}(\sqrt{n}+\sigma t)^2-\frac{n}{2}-\sigma t(\sqrt{n}+r) \right)=\exp\left( \frac{1}{2}\sigma^2t^2-\sigma rt \right).
    \end{split}
    \end{equation}
Set $t=\frac{r}{\sigma}$ in \eqref{eq: sub-G markov}, then it follows
    \begin{equation*}
        \prob{\|X\|\geq \sigma(\sqrt{n}+r)} \leq e^{-\frac{r^2}{2}}.
    \end{equation*}
The proof is concluded by setting $\delta=e^{-\frac{r^2}{2}}$.
\end{proof}
Unlike \eqref{eq: thm1}, the concentration inequality \eqref{eq:concen2} does not depend on $\varepsilon$ due to the additional optimization step \eqref{eq:opteps}. In practice, applying Theorem \ref{thm: norm concentration} or Theorem \ref{thm3} can depend on the value of $n$ and $\delta$. 

Interestingly, Theorem \ref{thm3} coincides with a concentration bound established in \cite[Remark 6]{zhivotovskiy2024dimension} based on variational inequality. Theorem \ref{thm3} is only slightly worse than the tightest existing result of sub-Gaussian norm concentration stated in \cite[Theorem 1]{hsu2012tail} which states that $\prob{\|X\|\geq \sigma\sqrt{n+2\sqrt{n\log(1/\delta)}+2\log(1/\delta)}}\le \delta$. It is worth pointing out that the proof in \cite{hsu2012tail} is based on bounding $\mbE_X\mbE_{\ell\sim \mathcal{N}(0, I)}(e^{\lambda\innerp{\ell,X}})$, which is an average of the MGF over the standard Gaussian distribution rather than the uniform distribution on $\mS^{n-1}$ as in AMGF.

\section{Generalization to Sub-Gaussian Random Matrices}
In this section, we extend our AMGF-based method to study the norm concentration property of sub-Gaussian matrices. In particular, we focus on the operator norm, which plays a pivotal role in applications related to random matrix theory. Recall that a random matrix $A\in\R^{m\times n}$ is said to be sub-Gaussian with variance proxy $\sigma^2$ if $\mbE(A)=0$ and for any $\lambda\in\R$ \cite[Section 1.2]{rigollet2023high}
\begin{equation} \label{eq: def subG mat}
    \expectw{A}{e^{\lambda u^\top Av}}\leq e^{\frac{\lambda^2\sigma^2}{2}},\quad\forall u\in\mS^{m-1},~\forall v\in\mS^{n-1}.
\end{equation}

Following Definition \ref{def: amgf}, we define the AMGF of a random matrix $A$ as $\Phi_A(\lambda)=\mbE_A(\Phi_{m,n}(\lambda A))$, where the energy function
\begin{equation}\label{eq: def mat AMGF}
    \Phi_{m,n}(\lambda A):=\expectw{u\sim\mS^{m-1},\\ v\sim\mS^{n-1}}{e^{\lambda u^\top Av}},\quad A\in\R^{m\times n}.
\end{equation}
Similar to the vector setting, the property of exponential growth holds for $\Phi_{m,n}(\lambda A)$.
\begin{lemma} \label{lemma: AMGF_mat}
    Given $A\in\R^{m\times n}$, for any $\varepsilon\in(0,1)$, $\Phi_{m,n}(\lambda A)$ satisfies
    \begin{equation}
        \Phi_{m,n}(\lambda A)\geq (1-\varepsilon^2)^{\frac{m+n}{2}}e^{\varepsilon^2\|\lambda A\|}.
    \end{equation}
\end{lemma} 

\begin{proof}
Without loss of generality, assume $m\leq n$. Consider the singular value decomposition $A=U\Sigma V$, where $U\in\R^{m\times m}$ and $V\in\R^{n\times n}$ are unitary matrices, and $\Sigma=\begin{bmatrix} \text{diag}(\sigma_1,\dots,\sigma_m) &  0_{m\times(n-m)}\end{bmatrix}$ with singular values $\sigma_1,\dots,\sigma_m$ arranged in a descending order. Note the operator norm of $A$ satisfies $\|A\|=\sigma_1$. Since for any unit vectors $u\in\mS^{m-1}$ and $v\in\mS^{n-1}$, $u^\top U$ and $Vv$ remain unit vectors, $\Phi_{m,n}(A)$ can be expressed as
    \begin{equation} \label{eq: bound mat AMGF step 1}
    \begin{split}
        &\Phi_{m,n}(\lambda A)=\expectw{v\sim\mS^{n-1}}{\expectw{u\sim\mS^{m-1}}{e^{\lambda\innerp{u,\Sigma v}}}}\\
        =&\,\,\expectw{v\sim\mS^{n-1}}{\Phi_m(\lambda \Sigma v)}\geq (1-\varepsilon^2)^{\frac{m}{2}}\expectw{v\sim\mS^{n-1}}{e^{\varepsilon \|\lambda\Sigma v\|}},
    \end{split}
    \end{equation}
    where the last ``$\geq$'' follows from Lemma \ref{lemma: AMGF_2}. 
    Since $\|\Sigma v\|\geq \sigma_1v_1=\|A\|\innerp{\ell_1,v}$ where $\ell_1$ is the unit vector along the first coordinate,
    \eqref{eq: bound mat AMGF step 1} can be further bounded as
    \begin{equation}
        \begin{split}
            &\Phi_{m,n}(\lambda A)\geq (1-\varepsilon^2)^{\frac{m}{2}}\expectw{v\sim\mS^{n-1}}{e^{\varepsilon \|\lambda\Sigma v\|}} \\
            \geq & (1-\varepsilon^2)^{\frac{m}{2}}\expectw{v\sim\mS^{n-1}}{e^{\varepsilon\|\lambda A\| \innerp{\ell_1,v}}}=(1-\varepsilon^2)^{\frac{m}{2}}\Phi_n(\varepsilon\|\lambda A\|\ell_1) \\
            \geq & (1-\varepsilon^2)^{\frac{m+n}{2}}e^{\varepsilon^2\|\lambda A\|}.
        \end{split}
    \end{equation}
\end{proof}

Based on Lemma \ref{lemma: AMGF_mat}, we present the following norm concentration inequality for sub-Gaussian matrices.

\begin{thm} \label{thm: mat norm}
    Let $A\in\R^{m\times n}$ be a sub-Gaussian matrix with variance proxy $\sigma^2$, then for any $r>0$:
    \begin{equation*}
         \prob{\|A\|\leq \sigma\sqrt{\frac{\log\frac{1}{1-\varepsilon^2}}{\varepsilon^4} (m+n)+\frac{2}{\varepsilon^4}\log\frac{1}{\delta}}}\ge 1-\delta.
    \end{equation*}
\end{thm}

\begin{proof}
    By Lemma \ref{lemma: AMGF_mat}, $\Phi_A(\lambda)\geq (1-\varepsilon^2)^{\frac{m+n}{2}}\expectw{A}{e^{\varepsilon^2\|\lambda A\|}}$.
    By the definition of sub-Gaussian matrix, $\Phi_A(\lambda)\leq e^{\frac{\lambda^2\sigma^2}{2}}$. Combining the upper and lower bounds on $\Phi_A(\lambda)$, we get
    \begin{equation}\label{eq: MGF_|A| upper}
        \expectw{A}{e^{\varepsilon^2\|\lambda A\|}}\leq (1-\varepsilon^2)^{-\frac{m+n}{2}}\exp\left(\frac{\lambda^2\sigma^2}{2}\right).
    \end{equation}
    Define $t=\varepsilon^2|\lambda|$, then \eqref{eq: MGF_|A| upper} becomes
    \begin{equation}\label{eq: MGF_t|A| upper}
        \expectw{A}{e^{t\| A\|}}\leq (1-\varepsilon^2)^{-\frac{m+n}{2}}\exp\left(\frac{\sigma^2t^2}{2\varepsilon^4}\right).
    \end{equation}
    By Markov's inequality, \eqref{eq: MGF_t|A| upper} implies that for any $t>0$:
    \begin{equation}\label{eq: sub-G ||A|| markov}
    \begin{split}
         \prob{\|A\|\geq r} \leq \frac{\expectw{A}{e^{t\| A\|}}}{e^{tr}}\leq  (1-\varepsilon^2)^{-\frac{m+n}{2}}\exp\left(\frac{\sigma^2t^2}{2\varepsilon^4}-rt\right).
    \end{split}
    \end{equation}
    Set $t=\frac{\varepsilon^4r}{\sigma^2}$ in \eqref{eq: sub-G ||A|| markov}, which is the minimizer of the exponent $\frac{\sigma^2t^2}{2\varepsilon^4}-rt$, then it follows that
    \begin{equation} \label{eq: mat norm <r}
        \prob{\|A\|\geq r} \leq (1-\varepsilon^2)^{-\frac{m+n}{2}}\exp\left(-\frac{\varepsilon^4r^2}{2\sigma^2}\right).
    \end{equation}
    Denote $\delta=(1-\varepsilon^2)^{-\frac{m+n}{2}}\exp\left(-\frac{\varepsilon^4r^2}{2\sigma^2}\right)$, then Theorem \ref{thm: mat norm} follows by plugging this expression into \eqref{eq: mat norm <r}.
\end{proof}

We point out that in Theorem \ref{thm: mat norm}, we do not require the elements of $A$ to be independent, as assumed in the proof of the Sub-Gaussian norm concentration based on the $\varepsilon$-net method \cite[Chapter 4]{vershynin2018high}.

\section{Conclusion} \label{sec: conclusion} 
This paper presents an alternative proof of the sub-Gaussian norm concentration inequality that applies to both random vectors and random matrices. The key to our proof is a modified MGF dubbed AMGF. The AMGF depends solely on the distribution of the norm of a random vector/matrix $X$, making it particularly suitable for analyzing the concentration properties of $\|X\|$. Unlike existing methods that rely on the union bound, our proof directly addresses $\prob{\|X\|>r}$, leading to a more refined analysis. Our method can also be applied to study the norm concentration of some other classes of distributions such as sub-exponential distributions. 
Beyond its immediate application to the norm concentration, the AMGF and the associated energy function $\Phi_{n}$ hold the potential for broader and lasting contributions to probability theory and related fields.

\bibliographystyle{IEEEtran}
\bibliography{main}

\end{document}